\documentclass[reqno,10pt]{amsart}
\usepackage{amscd,amssymb,amsmath,color}
\usepackage{latexsym}
\usepackage[all]{xy}
\usepackage{defs}
\usepackage[colorlinks=true]{hyperref}

\def\rmlog{{\rm log}}
\def\..{{,\dots,}}

\def\rt{{rt}}
\def\tw{{w}}
\def\sw{{sw}}
\def\hu{{hu}}
\def\Loc{{\rm Val}}

\begin{document}

\author{Michael Temkin}
\title{Tame distillation and desingularization by $p$-alterations}

\address{Einstein Institute of Mathematics, The Hebrew University of Jerusalem, Giv'at Ram, Jerusalem, 91904, Israel}
\email{temkin@math.huji.ac.il}
\thanks{This work was supported by the Israel Science Foundation (grant No. 1159/15).}

\begin{abstract}
We strengthen Gabber's $l'$-alteration theorem by avoiding all pri\-mes invertible on a scheme. In particular, we prove that any scheme $X$ of finite type over a quasi-excellent threefold can be desingularized by a $\cha(X)$-alteration, i.e. an alteration whose order is only divisible by primes non-invertible on $X$. The main new ingredient in the proof is a tame distillation theorem asserting that, after enlarging, any alteration of $X$ can be split into a composition of a tame Galois alteration and a $\cha(X)$-alteration. The proof of the distillation theorem is based on the following tameness theorem that we deduce from a theorem of M. Pank: if a valued field $k$ of residue characteristic $p$ has no non-trivial $p$-extensions then any algebraic extension $l/k$ is tame.
\end{abstract}

\maketitle

\section{Introduction}

\subsection{Background}
This paper falls within the area of resolution of singularities by alterations, so we start with a brief review of known altered desingularization results.

\subsubsection{de Jong's theorems}
In \cite[Theorem~4.1]{dJ} Johan de Jong proved that, regardless of the characteristic of the ground field, an integral variety $X$ can be desingularized by an {\em alteration} $b\:X'\to X$, i.e. a proper dominant generically finite morphism between integral schemes. In addition, given a closed subset $Z\subsetneq X$ one can achieve that $Z'=b^{-1}(Z)$ is a simple normal crossings (snc) divisor, and $f$ can be chosen {\em $G$-Galois} in the sense that the alteration $X'/G\to X$ is generically radicial, where $G=\Aut_X(X')$, see \cite[Theorem~7.3]{dJ}. de Jong's altered desingularization was the first resolution result that applies in such generality and it immediately found numerous applications. Also, de Jong proved an altered version of semistable reduction for an integral scheme $X$ over an excellent curve $S$, see \cite[Theorem~8.2]{dJ}. The latter can be viewed as altered desingularization of morphisms $f\:X\to S$ with $S$ a curve.

\subsubsection{Gabber's $l'$-strengthening}
The results of \cite{dJ} were strengthened a few times by de Jong and others, see e.g. \cite{dJ1}, \cite[Section~4]{Vidal} and \cite{X}. The most recent and powerful advance is Gabber's $l'$-altered desingularization: it is proved in \cite[Theorems~2.1 and 2.4]{X} that if $l$ is a prime invertible on $X$ then one can achieve that the degree $[k(X'):k(X)]$ of the alteration $b$ is not divisible by $l$ both in the altered desingularization and altered semistable reduction theorems. This extended the field of applications of altered desingularization to cohomology theories with coefficients where $l$ is not inverted, for example, $\bfZ/l\bfZ$ or $\bfZ_l$.

\subsubsection{$l$'-altered desingularization of morphisms}
When working out Gabber's proof, Luc Illusie and the author discovered a more general statement, which can be viewed as $l'$-altered desingularization of arbitrary finite type morphisms and contains the two theorems of Gabber as the particular cases with $\dim(S)$ equal to $0$ and $1$, see \cite[Theorem~3.5]{X}. In addition, in the characteristic zero case the same method provided an actual desingularization of morphisms \cite[Theorem~3.9]{X}, which extended the semistable reduction theorem of Abramovich and Karu \cite[Theorem~0.3]{Abramovic-Karu} from varieties to arbitrary quasi-excellent schemes.

\subsection{Main desingularization results}
Now, let us describe the advance this paper makes in the theory of altered desingularization.

\subsubsection{$\cha(X)$-altered desingularization}
Given a scheme $X$, by $\cha(X)$ we denote the set of all primes $p$ with $p=\cha(k(x))$ for some $x\in X$. Let $\calP$ be a set of primes. We say that $n\in\bfN$ is a {\em $\calP$-number} if all its prime divisors lie in $\calP$. By a {\em $\calP$-alteration} we mean an alteration whose degree is a $\calP$-number. The main goal of this paper is to strengthen the $l'$-altered desingularization of morphisms of \cite[Theorem~3.5]{X} so that the resolving alteration $b\:X'\to X$ is a $\cha(X)$-alteration. In particular, this unifies almost all previous results, including the characteristic zero case. The only thing which is lost is a control on the Galois properties of the alteration, for example, we do not achieve that $b$ is Galois.

\subsubsection{Main theorem}
Our main altered desingularization result is Theorem \ref{mainth}. Its formulation involves some terminology introduced in Section~\ref{desingsec}. Loosely speaking, the theorem asserts that if $f\:X\to S$ is a morphism of finite type and any alteration of $S$ can be desingularized by a $\cha(S)$-alteration then both $X$ and $f$ can be desingularized by a $\cha(X)$-alteration.

For any practical application, we should start with a class of schemes $S$ that satisfy the desingularization property as above. By a recent theorem of Cossart-Piltant, see \cite{CP}, any qe (i.e. quasi-excellent) threefold admits a desingularization, so we can take $S$ to be a qe threefold. This gives rise to two main consequences of Theorem \ref{mainth} that we are going to formulate. The proofs will be given in the end of Section~\ref{desingsec}.

\begin{rem}\label{CPrem}
We would like to clarify the current status of desingularization. There are published proofs for the case of qe surfaces, see e.g. \cite{Lip}, and for the case of $k$-varieties of dimension at most 3 under the minor assumption that the imperfection rank of $k$ is finite (i.e., if $p=\cha(k)>0$ then $[k:k^p]<\infty$), see \cite{CP0}. In the preprint \cite{CP}, Cossart and Piltant extend their method from \cite{CP0} to arbitrary qe threefolds, and experts in the desingularization theory think that the proof is correct. If the reader prefers not to use unpublished results, he should use the class of qe surfaces and 3-dimensional varieties instead of qe 3-folds.
\end{rem}

\subsubsection{Absolute desingularization}
For generality of formulations we consider non-integral schemes, so we refer to \S\ref{altsec} for a general definition of alterations in this context. Note, however, that the general case is not essentially stronger than the case of integral $X$ and $S$, in which it suffices to consider alterations as defined earlier.

\begin{theor}\label{absth}
Let $X$ be a scheme with a nowhere dense closed subset $Z$ and assume that $X$ admits a morphism of finite type to a qe scheme $S$ with $\dim(S)\le 3$. Then there exists a projective $\cha(X)$-alteration $b\:X'\to X$ with a regular source such that $Z'=b^{-1}(Z)$ is an snc divisor. Moreover, if $S=\Spec(k)$, where $k$ is a perfect field, then the alteration $b$ can be chosen separable.
\end{theor}

\begin{rem}
(i) This result is based on \cite{CP}. Until the latter is published, the reader may wish to only consider the case when either $\dim(S)\le 2$ or $S$ is of dimension $3$ over a field $k$ of a finite imperfection rank. In this case, the assertion is based on published results, see Remark~\ref{CPrem}.

(ii) The same remark applies to Theorem~\ref{relth} below.
\end{rem}

\subsubsection{Log smoothness}
To formulate our main result on desingularization of morphisms we will need a little bit of log geometry. We will provide references to general definitions, and make this explicit in the particular case we will use. Log schemes and log smooth morphisms between them are defined in \cite{Kato-log}. Log regular log schemes are defined in \cite{Kato-toric}, see also \cite{Niziol}. Given a scheme $X$ with a closed subscheme $Z$ consider the log structure $\calM_X=\calO_X\cap i_*\calO_U^\times$, where $i\:U=X\setminus Z\into X$ is the complementary open immersion. In general, it is not even fine. For example, this is the case of the cuspidal curve $X=\Spec(k[x^2,x^3])$ with $Z$ given by $x=0$: the stalk $\oM_{X,Z}$ is easily seen to be not finitely generated since it contains independent elements $x^2+ax^3$ with $a\in k$ (see also \cite[Remark~3.2.7(iii)]{VIII}). Nevertheless, in the situation we are going to use it will always be the case that $\calM_X$ is fs and the fs log scheme $(X,\calM_X)$ is log regular. For example, this happens when $X$ is regular and $Z$ is snc. For shortness, we will denote the log scheme $(X,\calM_X)$ by $(X,Z)$.

If $Z\into X$ and $T\into Y$ are closed immersions then any morphism $f\:X\to Y$ such that $T\times_YX\into Z$ induces a morphism of log schemes $(X,Z)\to (Y,T)$. Assume that $X$ and $Y$ are regular and $Z$ and $T$ are snc. Then it follows easily from \cite[Theorem~3.5]{Kato-log} that $f\:(X,Z)\to(Y,T)$ is log smooth if and only if for any point $x\in X$ with $y=f(x)$ there exist regular parameters $t_1\..t_s\in\calO_{Y,y}$ defining $T$ at $y$ and $z_1\..z_r\in\calO_{X,x}$ defining $Z$ at $x$ such that:
\begin{itemize}
\item[(1)]
$f^*(t_i)=\prod_{j=1}^rz_j^{l_{ij}}$ for $1\le i\le s$,
\item[(2)] both the kernel and the torsion part of the cokernel of the induced map $l\:\bfZ^s\to\bfZ^r$ are finite of order invertible in $k(y)$,
\item[(3)]
for small enough neighborhoods $X_0$ of $x$ and $Y_0$ of $y$ the induced morphism $X_0\to Y_0\times_{\Spec(\bfZ[t_1\..t_s])}\Spec(\bfZ[z_1\..z_r])$ is smooth.
\end{itemize}

\subsubsection{Desingularization of morphisms}
A morphism is {\em maximally dominating} if it takes generic points to generic points.

\begin{theor}\label{relth}
Let $f{\colon}X\to S$ be a maximally dominating morphism of finite type, let $Z\subset X$ be a nowhere dense closed subset, and assume that $X$ is separated and $S$ admits a morphism of finite type to a noetherian qe threefold. Then,

(i) There exist $\cha(S)$-alterations with regular sources $a{\colon}S'\to S$ and $b{\colon}X'\to X$, a quasi-projective morphism $f'{\colon}X'\to S'$ compatible with $f$ and snc divisors $W'\subset S'$ and $Z'\subset X'$ such that $a$ and $b$ are projective, $Z'=b^{-1}(Z)\cup f'^{-1}(W')$, and the morphism $(X',Z')\to(S',W')$ is log smooth.

(ii) If $S=\Spec(k)$, where $k$ is a perfect field, then one can achieve in addition to (i) that $a$ is an isomorphism and the alteration $b$ is separable.
\end{theor}

Also, it is possible to remove the separatedness assumption on $X$ in Theorem~\ref{relth} at cost of considering a pseudo-projective $f'$, see Theorem~\ref{mainth}.

\subsection{The method and distillation theorems}
Our proof of Theorem~\ref{mainth} is very close to the proof of its predecessor \cite[Theorem~3.5]{X}, with Theorem~\ref{altdistth} being the only new ingredient. Theorem~\ref{altdistth} constructs tame distillations of alterations, and its proof occupies most of the paper. First, we prove analogous distillation theorems for extensions of fields, and then use the latter to prove the distillation theorem for alterations. Let us now describe our method, distillation theorems, and the structure of the paper.

\subsubsection{Distillation}
To illustrate the term ``distillation", let us consider the classical theory of separable and inseparable field extensions. Any algebraic extension $l/k$ canonically factors as $l/l_s/k$, where $l/l_s$ is purely inseparable and $l_s/k$ is separable. By a separable distillation of $l/k$ we mean an opposite factoring $l/l_i/k$, in which the bottom extension is purely inseparable and the top one is separable. It is non-canonical and may not exist, see e.g. \cite[V.145, \S7, Exercices 1--3]{BouAlg} for various related examples. In fact, existence of a separable distillation is equivalent to splitting the tower $l/l_s/k$, i.e. splitting $l$ as $l_s\otimes_kl_i$. The algebraic closure $k^a/k$ possesses a separable distillation $k^a/k^{1/p^\infty}/k$ (which is even canonical) and it follows easily that for any finite extension $l/k$ there exists a larger finite extension $l'/k$ possessing a separable distillation. Let us call this simple result separable distillation theorem.

\subsubsection{Tame distillation for valued fields}
Wild ramification in the theory of valued fields is an analogue of inseparability in the usual field theory. In particular, both notions coincide when the valuation is trivial. It turns out that the separable distillation theorem has an analogue in the theory of valued fields, a tame distillation theorem, though the latter result is already not so simple. If a valued field $k$ is henselian, this is essentially a theorem of M. Pank: the tower $k^a/k^t/k$ splits, where $k^t$ is the tame closure of $k$. In fact, $k^a=k^t\otimes_kk_\tw$, where $k_\tw/k$ is any maximal purely wild extension. We recall Pank's theorem and discuss its various reformulations in Section~\ref{panksec}.

It was known to F. Pop (unpublished), but perhaps not to other experts on valuation theory, that Pank's theorem extends to non-henselian valued fields. In order to minimize the use of valuation theory in this paper, we only establish the splitting theorem when the height of $k$ is one, see Theorem~\ref{htonesplit}. It is deduced from Pank's theorem via a decompletion argument. Using induction on height one can then extend the generalized Pank's theorem to valued fields of arbitrary height, but this is rather technical and will be worked out elsewhere.

For our purposes, it will be enough to establish in Theorem~\ref{tameth} the following tameness theorem: if $k$ has no non-trivial $p$-extensions then any algebraic extension $l/k$ is tame. In fact, tameness theorem is weaker than the generalized Pank's theorem, so the case of height one follows from Theorem~\ref{htonesplit}, and the general case will be established by induction on height (which is less technical in this case).

\subsubsection{Riemann-Zariski spaces and globalization}
Let $L/K$ be an algebraic extension of fields. Our next aim is to obtain tame distillation simultaneously for a set $S$ of valuations on $L$. Naturally, we should assume that $S$ is quasi-compact in $\Val_L$ (see Section~\ref{rzsec}) and $L/K$ is finite. By $\cha(S)$ we denote the set of non-zero residue characteristics of $S$ and we say that $L/K$ is $S$-tame if $L/K$ is tame for any valuation of $S$. Tame distillation for abstract fields is proved in Theorem~\ref{extdistth}; it implies, in particular, that there exists a finite extension $L'/L$ such that $L'/K$ splits into a composition of a $\cha(S)$-extension $K'/K$ and an $S'$-tame extension $L'/K'$, where $S'$ is the preimage of $S$ in $\Val_{L'}$.

Our proof of Theorem~\ref{extdistth} is based on Theorem~\ref{tameth} and the following facts about Riemann-Zariski spaces: $\Val(L)$ is quasi-compact, the map $\psi_{L/K}:\Val(L)\to\Val(K)$ induced by $L/K$ is open, and the tame locus of $\psi_{L/K}$ is open.

\subsubsection{Tame distillation of alterations}
A further globalization of tame distillation is done in Theorem~\ref{altdistth} asserting that any alteration $b\:Y\to X$ can be enlarged to a Galois alteration $Y'\to X$ that splits into a composition of a tame Galois alteration $Y'\to X'$ and a $\cha(X)$-alteration $X'\to X$. The main point of its proof is to apply the tame distillation theorem to $k(Y)/k(X)$ with $S$ being the set of all valuations of $k(Y)$ with center on $X$.

\subsubsection{Main theorem}
Finally, let us briefly explain how tame distillation of alterations is used to obtain $\cha(X)$-altered desingularization. The main idea of Gabber's construction of $l'$-alterations is as follows: first obtain a $G$-Galois alteration $X'\to X$ by de Jong's methods, and then fix an $l$-Sylow subgroup $G_l\subseteq G$ and pass to $X/G_l\to X$. Finally, the action of $G_l$ is tame, hence if $X'$ is regular then the singularities of $X'/G_l$ can be resolved by Gabber's modification theorem \cite[Theorem~1.1]{X}. This scheme is used both in the proof of Gabber's theorems \cite[2.1 and 2.4]{X}, and in the proof of their refinement \cite[Theorem~3.5]{X}.

If there exists a subgroup $H\subseteq G$ such that $H$ acts tamely on $X'$ and any prime dividing $|G/H|$ is in $\cha(X)$, then we can obtain a $\cha(X)$-altered desingularization in the same manner. By the tame distillation theorem such an $H$ exists after enlarging the alteration. Fortunately, this does not affect the proof of \cite[Theorem~3.5]{X}, which extends almost verbatim once Sylow subgroups are replaced with the subgroups produced by tame distillation. 

\subsection{Acknowledgments}
I wish to thank F. Pop and F.-V. Kuhlmann for answering my questions about Pank's theorem. In particular, Florian Pop informed me about his unpublished generalization of Pank's theorem to the non-henselian case. I am grateful to L. Illusie, V. Varshavsky and I. Tyomkin for useful discussions. Also I wish to thank G. Pohl and the anonymous referee for pointing out some gaps and inaccuracies in the first version of the paper.

\setcounter{tocdepth}{1}
\tableofcontents

\section{Splitting off wild extensions}

\subsection{Field extensions}

\subsubsection{$\calP$-extensions}\label{calPsec}
First, we introduce a terminology that extends the notions of $l'$-extensions and $\cha(X)$-extensions from the introduction. Throughout the text, we will often denote by $\calP$ a set of primes. Then the set of all primes not contained in $\calP$ will be denoted $\calP'$. A typical example of $\calP$ is the set $\cha(X)$ of all non-zero residue characteristics of a scheme $X$. Also, to any natural number $n\ge 1$ we associate the set $\calP(n)=(\cha(\Spec\,\bfZ[\frac{1}{n}]))'$ of its prime divisors. For example, $\cha(\Spec\,\bfQ)=\calP(1)=\emptyset$ and $\cha(\Spec\,\bfZ_{(2)})=\calP(8)=\{2\}$.

A {\em $\calP$-group} is a finite group $G$ such that $|G|$ is a $\calP$-number, that is, any prime divisor of $|G|$ lies in $\calP$. Analogously, an algebraic field extension $L/K$ is called a {\em $\calP$-extension} if $[L':K]$ is a $\calP$-number for any finite subextension $L'/K$. Note that we neither require that $L/K$ is Galois nor that its Galois closure is a $\calP$-extension. For shortness, we will say $n$-extension, $n'$-group, etc., instead of $\calP(n)$-extension, $\calP(n)'$-group, etc. For example, $1'$-extension is an arbitrary algebraic extension.

The class of $\calP$-extensions is transitive (i.e. if $F/L$ and $L/K$ are $\calP$-extensions then $F/K$ is so) and closed under taking subextensions and filtered union: these claims reduce to the case of finite extensions and then follows from the fact that for a tower of finite extensions $F/L/K$, the degree $[F:K]=[F:L][L:K]$ is a $\calP$-number if and only if both $[F:L]$ and $[L:K]$ are $\calP$-numbers. In particular, Zorn's lemma implies that any field $K$ possesses a maximal $\calP$-extension $L/K$ and then $L$ has no non-trivial $\calP$-extensions. Any field without non-trivial $\calP$-extensions will be called {\em $\calP$-closed}. We warn the reader that a finite extension of a $\calP$-closed field does not have to be $\calP$-closed.

\subsubsection{Split towers}
We say that a tower of algebraic field extensions $L/K/k$ is {\em split} if there exists an extension $l/k$ such that $L=l\otimes_kK$. In other words, $l$ is linearly disjoint from $K$ over $k$ and the composite $lK$ coincides with $L$.

\subsection{Basic ramification theory: the henselian case}\label{henssec}
Basic ramification theory for henselian valued fields is classical, so we only recall some definitions and basic facts without proofs. It is not so easy to find a good reference, but the reader can consult \cite[Section~6.2]{Gabber-Ramero} or a book in preparation \cite{Kuhl} by F.-V. Kuhlmann.

\subsubsection{Valued fields}
We will use multiplicative notation in the theory of valued fields. In particular, by a {\em valued field} we mean a field $k$ provided with a non-archimedean valuation $|\ |\:k\to\{0\}\coprod\Gamma$. The ring of integers and the residue field will be denoted $\kcirc$ and $\tilk$, respectively.

\subsubsection{Henselian valued fields}
A valued field $k$ is called {\em henselian} if the local ring $\kcirc$ is henselian. This happens if and only if the valuation of $k$ extends uniquely to any finite extension $l/k$. Throughout Section~\ref{henssec}, $k$ is a henselian valued field and $p=\expchar(\tilk)$ is the residual characteristic exponent of $k$.

\subsubsection{Unramified extensions}
By $k^u$ we denote the maximal unramified extension of $k$, i.e. $(k^u)^\circ$ is the strict henselization of $\kcirc$. Note that $k^u/k$ is a Galois extension and $G_{k^u/k}=G_{(\tilk)^s/\tilk}$.

\subsubsection{Totally ramified extensions}
An algebraic extension $l/k$ is called {\em totally ramified} if it is linearly disjoint from $k^u$. This happens if and only if $\till/\tilk$ is purely inseparable (e.g. $\tilk=\till$).

\subsubsection{Tame extensions}
Any algebraic extension $l/k$ splits uniquely into a tower $l/m/k$ such that $l/m$ is totally ramified and $m/k$ is unramified. If $l/m$ is a $p'$-extension then $l/k$ is called {\em tame}. A tame extension $l/k$ is controlled pretty well by the separable residue field extension $\till/\tilk$ and the $p'$-torsion group $G=|l^\times|/|k^\times|$. In particular, $[l:k]=[\till:\tilk]\cdot|G|$.

\subsubsection{The tame closure}
There is a unique maximal tame extension $k^t/k$ and the wild inertia group $W_k=G_{k^s/k^t}$ is a pro-$p$-group, which is the only Sylow pro-$p$-subgroup of the inertia group $I_k=G_{k^s/k^u}$. In particular, $k^t/k^u$ is Galois, and Kummer's theory implies that $G_{k^t/k^u}=\Hom(|(k^t)^\times|/|k^\times|,\mu_{k^u})$. Note also that $|(k^t)^\times|$ is the $p'$-divisible envelope of $|k^\times|$, i.e. it is obtained from $|k^\times|$ by extracting all roots of prime-to-$p$ orders.

\subsubsection{Purely wild extensions}
An algebraic extension $l/k$ is called purely wild if it is linearly disjoint from $k^t$. The following conditions are equivalent: (1) $l/k$ is purely wild, (2) $l/k$ is a totally ramified $p$-extension, (3) $\till/\tilk$ is purely inseparable and $|l^\times|/|k^\times|$ is a $p$-group.

\subsection{Basic ramification theory: the general case}\label{nonhenssec}
Now, let us discuss the case when $k$ is an arbitrary valued field. The results we recall in this section are still well known, but it is harder to find a single reference to this material and some terminology is not standardized. In particular, our notions of (strictly) unramified extensions and henselian extensions are not standard.

\subsubsection{Henselization}
Given a valued field $k$, by $k^h$ we denote its {\em henselization}, i.e. the valued field whose ring of integers is the henselization of $\kcirc$.

\begin{lem}\label{henslem}
Assume that $k$ is a valued field, $l/k$ is a finite field extension, and $l_1\..l_n$ is the list of all non-isomorphic valued extensions of $k$ whose underlying field is $l$. Then $l\otimes_kk^h=\prod_{i=1}^nl_i^h$.
\end{lem}
\begin{proof}
Let $\lcirc=\Nor_l(\kcirc)$ be the integral closure of $\kcirc$ in $l$. By \cite[Ch.VI, \S8, n.6, Proposition~6]{BouComAlg}, $\lcirc$ is a semilocal ring whose localizations at the maximal ideals are the rings $l_i^\circ$. By \cite[$\rm {IV}_4$, Proposition~18.6.8]{ega} the ring $(\lcirc)^h:=\prod_{i=1}^n(\lcirc_i)^h$ is canonically isomorphic to $\lcirc\otimes_{\kcirc}(\kcirc)^h$. The lemma follows by tensoring with $k$.
\end{proof}

\subsubsection{Extensions of valuation rings}
We will also need the following simple result.

\begin{lem}\label{extlem}
Assume that $k$ is a valued field, $l/k$ is a finite field extension, and $A$ is an integrally closed $\kcirc$-subalgebra of $l$. Then all localizations of $A$ are valuation rings (i.e. $A$ is a Pr\"ufer ring).
\end{lem}
\begin{proof}
Decreasing $l$ we can assume that $l=\Frac(A)$. Any localization $A_q$ contains a localization $C$ of $\Nor_l(\kcirc)$. It remains to note that $C$ is a valuation ring of $l$ by \cite[Ch.VI, \S8, n.6, Proposition~6]{BouComAlg}, and hence any $C$-subalgebra of $l$ is a valuation ring.
\end{proof}

\subsubsection{Basic ramification tower}
Setting $k^u=(k^h)^u$ and $k^t=(k^h)^t$ we obtain a {\em basic ramification tower} $k^s/k^t/k^u/k^h/k$ of $k$. Note that the bottom level $k^h/k$ is typically not normal unless $k$ is henselian. Many definitions and proofs are done by passing to henselizations.

\begin{rem}
The subgroup $D\subseteq G_{k^s/k}$ corresponding to $k^h$ is called the decomposition group of the valuation of $k^s$. Its conjugates correspond to other extensions of the valuation of $k$ to $k^s$.
\end{rem}

\subsubsection{Tame and unramified extensions}
Let $l/k$ be an algebraic extension of valued fields. We say that $l/k$ is {\em tame} (resp. {\em unramified}) if $l^h/k^h$ is so. In particular, $k^u$ (resp. $k^t$) is, indeed, the maximal unramified (resp. tame) extension of $k$. Note that $(k^u)^\circ$ is the strict henselization $(\kcirc)^\sh$ of $\kcirc$. Also, we say that $l/k$ is {\em strictly unramified} if it is unramified and $\tilk=\till$. In particular, $k^h/k$ is the maximal strictly unramified extension.

Recall that a local $\kcirc$-algebra $A$ is called {\em essentially \'etale} (resp. {\em strictly essentially \'etale}) if $A$ is a localization of an \'etale $\kcirc$-algebra (resp. and $\tilk$ is the residue field of $A$), see \cite[$\rm {IV}_4$, 18.6.1, 18.6.2]{ega}.

\begin{lem}\label{unramlem}
Consider the following conditions on a finite extension $l/k$ of valued fields:
\begin{itemize}
\item[(a)] $l/k$ is unramified (resp. strictly unramified),
\item[(b)]  $\lcirc/\kcirc$ is essentially \'etale (resp. strictly essentially \'etale),
\item[(c)] $\lcirc/\kcirc$ is \'etale (resp. strictly \'etale).
\end{itemize}
Then (a)$\Longleftrightarrow$(b)$\Longleftarrow$(c), and all three conditions are equivalent whenever the height of $k$ is finite.
\end{lem}
\begin{proof}
It suffices to establish the implications in the non-strict case. Obviously, (c)$\implies$(b). Assume that $\lcirc/\kcirc$ is essentially \'etale. Then $(\kcirc)^\sh=(\lcirc)^\sh$, and so $k^u=l^u$. Thus, (b)$\implies$(a). In addition, $\lcirc=A_m$ for an \'etale $\kcirc$-algebra $A$ with a maximal ideal $m$. If the height $h$ of $k$ is finite then $\Spec(\kcirc)$ is a finite set, and hence the set $\Spec(A)$ is also finite. It follows that the localization $A\to A_m$ can be achieved by inverting a single element. In particular, $\lcirc/\kcirc$ is \'etale, and we obtain that (b)$\implies$(c) whenever $h<\infty$.

It remains to prove that (a)$\implies$(b). If $l/k$ is unramified then $\lcirc\subseteq(\kcirc)^\sh$ and hence there exists a finite extension of valued fields $\oll/l$ such that $\oll^\circ/\kcirc$ is essentially \'etale, say, $\oll^\circ=\oA_\om$ for a $\kcirc$-\'etale algebra $\oA$ with a maximal ideal $\om$. Being $\kcirc$-\'etale, the ring $\oA$ is integrally closed. Hence $A=\oA\cap l$ is integrally closed too, and localizations of $A$ are valuation rings by Lemma~\ref{extlem}. Taking into account that $\oA$ is a domain, we obtain that $\oA$ is $A$-flat. Since $\oA$ is finitely generated over $\kcirc$, it is finitely generated over $A$ and hence finitely presented over $A$ by \cite[Corollaire~3.4.7]{RG}. Applying \cite[$\rm IV_4$, Lemme~17.7.5(iv)]{ega} to $\Spec(\oA)\to\Spec(A)\to\Spec(\kcirc)$, we obtain that $A$ is finitely presented over $\kcirc$, and hence $A$ is $\kcirc$-\'etale by \cite[$\rm IV_4$, Proposition~17.7.7]{ega}. It remains to notice that $\lcirc$ is the localization of $A$ along $m=\om\cap A$.
\end{proof}

\subsubsection{Tame fields}
A valued field $k$ is called {\em tame} if $k^a=k^t$. Clearly, this happens if and only if any algebraic extension $l/k$ is tame. Purely inseparable extensions are purely wild, hence any tame field is perfect. We also say that $k$ is {\em separably tame} if $k^s=k^t$.

\subsubsection{Henselian extensions}
We say that an algebraic extension of valued fields $l/k$ is {\em henselian} if the valuation of $k$ extends to the valuation of $l$ uniquely. For example, $k$ is henselian if and only if any algebraic extension $l/k$ is henselian.

\begin{lem}\label{hensextlem}
An algebraic extension $l/k$ is henselian if and only if $l$ and $k^h$ are linearly disjoint, and in this case $l\otimes_kk^h=l^h$.
\end{lem}
\begin{proof}
If $l/k$ is finite then this follows by applying Lemma~\ref{henslem} to $l/k$. The general case follows since henselizations are compatible with filtered colimits by \cite[$\rm {IV}_4$, Proposition~18.6.14(ii)]{ega}.
\end{proof}

\subsubsection{Totally ramified extensions}
As in the henselian case, $l/k$ is called {\em totally ramified} (resp. {\em purely wild}) if it is linearly disjoint from $k^u/k$ (resp. $k^t/k$). This happens if and only if $l/k$ is henselian and the extension $l^h/k^h$ is totally ramified (resp. purely wild). In particular, $l/k$ is purely wild if and only if it is a totally ramified $p$-extension.

\subsection{Splitting the basic ramification tower and Pank's theorem}\label{panksec}

\subsubsection{Non-canonical maximal extensions}
We have defined six basic properties of extensions of valued fields that form three ``disjoint pairs": strictly unramified versus henselian, unramified versus totally ramified, and tame versus purely wild.

\begin{lem}
Let $R$ be one of the following properties: (a) strictly unramified, (b) unramified, (c) tame, (d) henselian, (e) totally ramified, (f) purely wild.

(i) Let $m/l/k$ be a tower of algebraic extensions of valued fields. Then $m/k$ satisfies $R$ if and only if both $m/l$ and $l/k$ satisfy $R$.

(ii) A filtered union of $R$-extensions of $k$ is an $R$-extension.

(iii) Let $S$ be any set of the properties (a)--(f). Then any valued field $k$ possesses a maximal $S$-extension $k_S/k$.
\end{lem}
\begin{proof}
The first three properties are included for completeness. Claims (i) and (ii) hold for them because there even exists a unique maximal $R$-extension $k^h$, $k^u$ or $k^t$, respectively. The last three properties are equivalent to being linearly disjoint from $k^h$, $k^u$ and $k^t$, respectively. This observation and the fact that $l^h=lk^h$, $l^u=lk^u$ and $l^t=lk^t$ imply the assertions (i) and (ii) for (d), (e) and (f). The assertion (iii) follows from (i) and (ii) by Zorn's lemma.
\end{proof}

We will use lower indices to denote non-canonical maximal extensions, as opposed to $k^t$, $k^u$, etc. For example, $k_\tw$ will denote a maximal purely wild extension of $k$.

\subsubsection{Splitting the tower $k^t/k^u/k^h/k$}
Non-canonical maximal extensions can be used to split various levels in the basic ramification tower. We start with a simple case that will not be used in the sequel but illustrates the situation well.

\begin{lem}\label{tamesplit}
Assume that $k$ is a valued field. Let $k_\hu/k$ be a maximal henselian unramified extension and let $k_{\rt}/k$ be a maximal tame totally ramified extension. Then $k^h\otimes_kk_\hu=k^u$ and $k^u\otimes_kk_\rt=k^t$.
\end{lem}
\begin{proof}
Both pairs are linearly disjoint by definition hence we should only check that the extensions $k^u/k^hk_\hu$ and $k^t/k^uk_\rt$ are trivial. The first one is an unramified extension of henselian fields, hence it suffices to compare the residue fields. Both $\tilk^u$ and $\wt{k^hk_\hu}$ are separable over $\tilk$. Also, both $\tilk^u$ and $\tilk_\hu$ are separably closed because otherwise we could lift their non-trivial separable extension to a non-trivial henselian unramified extension of $k^u$ or $k_\hu$ contradicting their maximality. Thus, $\tilk^u=(\tilk)^s=\wt{k^hk_\hu}$ and hence $k^u=k^hk_\hu$.

The extension $k^t/k^uk_\rt$ is a tame extension of henselian fields, hence it suffices to compare the residue fields and the groups of values. The same argument as above shows that the residue fields coincide with $(\tilk)^s$. We remarked earlier that $H=|(k^t)^\times|$ is the $p'$-divisible envelope of $|k^\times|$. So, it remains to check that the inclusion $|(k_{\rt})^\times|\subseteq H$ is an equality. Indeed, if the inclusion were strict we would be able to enlarge $k_{\rt}$ by adjoining an appropriate root of its element.
\end{proof}

\subsubsection{Splitting $k^a/k^t/k$}
Analogously to Lemma~\ref{tamesplit}, any maximal purely wild extension $k_w/k$ splits $k^a/k^t/k$. However, this is more difficult to prove because wild extensions are not controlled well enough by the residue fields and groups of values. For henselian fields, this splitting is a theorem of M. Pank, and we will deduce that it also holds for arbitrary valued fields of height one. We will not need the general case and it will be published elsewhere.

\subsubsection{Reformulations}\label{equivsec}
Let $\calV$ be a class of valued fields closed with respect to algebraic extensions. For example, the class of henselian valued fields or the class of valued field of a given height. We claim that conditions (0)--(4) below are equivalent, where $k_w$ denotes a maximal purely wild extension of $k$.

\begin{itemize}
\item[(0)] A splitting of $k^a/k^t/k$: if $k\in\calV$ then $k^t\otimes_kk_w=k^a$ for some choice of $k_w$.
\item[(1)] Splittings of $k^a/k^t/k$: if $k\in\calV$ then $k^t\otimes_kk_w=k^a$ for any choice of $k_w$.
\item[(2)] Tame distillation: if $k\in\calV$ then $k_w$ is tame.
\item[(3)] Tame fields: a valued field $k\in\calV$ is tame if and only $k=k_w$.
\item[(4)] Splitting of the Galois group: for $k\in\calV$ the homomorphism $G_{k^s/k}\to G_{k^t/k}$ admits a section, i.e. $G_{k^s/k}$ is a semidirect product of the wild inertia group $W=G_{k^s/k^t}$ and the tame Galois group $G_{k^t/k}$.
\end{itemize}

Indeed, (1) obviously implies (0), and we have the following simple implications.

(0)$\implies$(4) Let $l=k^{1/p^\infty}$ be the perfection of $k$. The isomorphism $G_{l^a/l}=G_{k^s/k}$ respects the inertia subgroups, hence (4) follows by applying (0) to $l$.

(4)$\implies$(3) If $k_w=k$ then $k$ is perfect, i.e. $k^a=k^s$. In addition, the fixed field of the image of a section $G_{k^t/k}\to G_{k^s/k}$ is purely wild over $k$. Hence $k^s=k^t$.

(3)$\implies$(2) Note that $(k_w)_w=k_w$ and apply (3) to $k_w$.

(2)$\implies$(1). We should check that the extension $k^a$ coincides with its subfield $l=k^tk_w$. Since $k^t\subseteq l$, the extension $k^a/l$ is purely wild. Since $k_w$ is tame the extension $k^a/l$ is tame. So, $k^a=l$.

\begin{rem}\label{sepsplitrem}
Using (4) one can show in the same fashion that these conditions are also equivalent to the following separable analogues, where $k_\sw/k$ is a maximal separable purely wild extension:
\begin{itemize}
\item[(1)] If $k\in\calV$ then $k^t\otimes_kk_\sw=k^s$.
\item[(2)] If $k\in\calV$ then $k_\sw$ is separably tame.
\item[(3)] A valued field $k\in\calV$ is separably tame if and only if $k=k_\sw$.
\end{itemize}
\end{rem}

\subsubsection{Pank's theorem}
The following theorem was proved by M. Pank.

\begin{theor}\label{pankth}
Let $k_{\tw}$ be a maximal purely wild extension of a henselian valued field $k$. Then $k^t\otimes_kk_{\tw}=k^a$.
\end{theor}

At least two proofs are available in the literature. The proof of \cite[Theorem~2.1]{KPR} establishes condition (4) by a group-theoretic argument that uses Galois theory and group cohomology: one uses a pro-finite version of Schur-Zassenhaus theorem and the fact that $p$-Sylow subgroups of $G_{\tilk}$ are pro-$p$-free.

Another proof was suggested by Ershov, see \cite[Theorem~2]{Er}, and it establishes reformulation (3) by a direct valuation-theoretic argument. Similarly to our study of an extension $k_{\rt}$, it is easy to see that $\wt{k_{w}}$ is the perfection of $\tilk$ and $|(k_{w})^\times|$ is the $p$-divisible envelope of $|k^\times|$. This is seen already by considering $p$-th roots of elements. Thus the whole point of the proof is to control extensions $l/k$ with a non-trivial {\em defect} $d_{l/k}=[l:k]/(e_{l/k}f_{l/k})$. In fact, Ershov proves that if $l/k$ is tame and $l$ has a non-trivial {\em immediate} algebraic extension $l'$ (i.e. $e_{l'/l}=f_{l'/l}=1$) then $k$ has a non-trivial immediate algebraic extension too.

\subsection{Splitting theorem for valued fields of height one}\label{htonesec}

\subsubsection{Decompletion}
If $k$ is of height one then $k$ is dense in $k^h$ because $\hatk$ is henselian and hence contains $k^h$ (e.g., see \cite[Theorem~17.18]{Endler}). This allows to split separable extensions of $k^h$.

\begin{lem}\label{decomplem}
If $k$ is a valued field of height one then any tower $L/k^h/k$ with a separable $L/k^h$ splits. In other words, there exists a henselian extension $l/k$ such that $L=l^h=l\otimes_kk^h$.
\end{lem}
\begin{proof}
Assume first that $L/k^h$ is finite. Since $L/k^h$ is separable, $L=k^h[t]/(f(t))$ for a monic polynomial $f(t)$ by the primitive element theorem. By Krasner's lemma (e.g., see \cite[Lemma~16.8]{Endler}), $L=k^h[t]/(g(t))$ for any monic polynomial $g(t)$ of the same degree whose root is close enough to a root of $f(t)$. Since roots of monic polynomials depend on the coefficients continuously (e.g., see \cite[Theorem~2]{Brink}), we can achieve that $g(t)\in k[t]$ and then $l=k[t]/(g(t))$ is as required.

Now, let us establish the general case. Note that given a totally ordered family $\{l_i\}_{i\in I}$ of valued extensions of $k$ provided with a compatible family of embeddings of valued fields $l_i^h\into L$, we also obtain an embedding $(\cup_{i\in I}l_i)^h\into L$. By Zorn's lemma, there exists a maximal totally ordered family of henselian extensions $l_i/k$ provided with compatible embeddings $l_i^h\into L$, and we denote its maximal element $\cup_{i\in I}l_i$ by $l$. We claim that the inclusion $l^h\into L$ is an equality, and so $l$ is as required. Indeed, if $l^h\subsetneq L$ then there exists a finite subextension $L'/l^h$ and by the case of finite extensions, $L'/l^h/l$ can be split by an extension $l'/l$. Then $L'=(l')^h$, and we obtain that the family $\{l_i\}$ can be enlarged by adjoining $l'$. The contradiction concludes the proof.
\end{proof}

\begin{cor}\label{decomplcor}
Assume that $l/k$ is an algebraic extension of valued fields and $k$ is of height one. Then $l/k$ is a maximal separable purely wild extension if and only if $l/k$ is henselian and $l^h/k^h$ is a maximal separable purely wild extension.
\end{cor}
\begin{proof}
Let $R$ abbreviate ``non-trivial separable purely wild". We should prove that $l$ possesses an $R$-extension $l'$ if and only if $l^h$ possesses an $R$-extension $L'$. By Lemma~\ref{decomplem}, any separable algebraic extension $L'/l^h$ is of the form $l'l^h$ for a henselian extension $l'/l$. It remains to note that $l'l^h/l^h$ is an $R$-extension if and only if $l'/l$ is an $R$-extension.
\end{proof}

\subsubsection{Splitting}
Now, we can extend Pank's theorem to arbitrary valued fields of height one.

\begin{theor}\label{htonesplit}
Assume that $k$ is a valued field of height 1 and let $k_{\tw}/k$ be a maximal purely wild extension. Then $k^t\otimes_kk_{\tw}=k^a$.
\end{theor}
\begin{proof}
Assume, first that $k$ is perfect. Set $K=k^h$, then $K_{\tw}=k_{\tw}\otimes_kK$ is a maximal purely wild extension of $K$ by Corollary~\ref{decomplcor}. Hence $K^t\otimes_KK_{\tw}=K^a$ by Pank's theorem \ref{pankth}, and using that $K^a=k^a$ and $K^t=k^t$ we obtain that $k^t\otimes_kk_w=k^t\otimes_KK_w=k^a$.

In the general case, let $l=k^{1/p^\infty}$ be the perfection of $k$. Then $l^t=k^t\otimes_kl$ and hence $k^t\otimes_kl_w=l^t\otimes_ll_w=l^a$. This gives a splitting of $k^a/k^t/k$ and it remains to use the equivalence of (0) and (1) in \S\ref{equivsec}.
\end{proof}

\begin{cor}\label{tamecor}
Any $p$-closed valued field of height one is tame.
\end{cor}
\begin{proof}
Any extension $k_w/k$ is a $p$-extension. So, if $k$ is $p$-closed then $k_w=k$. If $k$ is also of height one then we deduce that $k^t=k^a$ by Theorem~\ref{htonesplit}.
\end{proof}

\subsection{Tameness of $p$-closed valued fields}
Our next aim is to extend Corollary~\ref{tamecor} to arbitrary valued fields. The basic tools are to approximate valued fields by valued fields of finite height and to represent valuations of finite height as compositions of valuations of height one.

\subsubsection{Composed valued fields}
If $k$ is a valued field and $\tilR$ is a valuation ring of $\tilk$ then the preimage of $\tilR$ under the map $\kcirc\to\tilk$ is a valuation ring $R$ of $k$ and one says that the valuation of $(k,R)$ is composed of the valuations of $(k,\kcirc)$ and $(\tilk,\tilR)$. Conversely, if $k$ is a valued field then for any prime ideal $n\subset\kcirc$ we have that $\kcirc$ is composed of the valuation rings $\kcirc_n$ and $\kcirc/n$.

If $k$ is of height $d<\infty$ then for any $0\le i\le d$ there exists a unique localization of $\kcirc$ of height $i$ that will be denoted $k_i^\circ$. Let $k_i=(k,k^\circ_i)$ denote the corresponding valued field and let $\tilk_i$ denote the residue field of $k_i$ with the valued field structure induced from $k$. In particular, the valuation of $k$ is composed of those of $k_i$ and $\tilk_i$.

\subsubsection{Tame extensions and composition}
For simplicity we only consider valued fields $k$ of finite height. Assume that $k$ is of height $d$ and set $p_i=\expcha(\tilk_i)$ and $p=\expcha(\tilk)$. In particular, $p=p_d$. If $X$ is a scheme and $x\in X$ is a generization of $y\in X$ then $\expcha(k(x))$ divides $\expcha(k(y))$. Since $\Spec(\kcirc)$ is a chain (ordered by generization) of $d+1$ points with residue fields $\tilk_i$, there exists $n$ with $0\le n\le d$ such that $p_i=1$ for $i<n$ and $p_i=p$ for $i\ge n$.

\begin{lem}\label{unramcompos}
Let $l/k$ be an algebraic extension of valued fields of height $d<\infty$ and $0\le i\le d$. Set $p=\expcha(\tilk)$ and $p_i=\expcha(\tilk_i)$. Then,

(i) $l/k$ is unramified if and only if both $l_i/k_i$ and $\till_i/\tilk_i$ are unramified.

(ii) $l/k$ is tame if and only if both $l_i/k_i$ and $\till_i/\tilk_i$ are tame and $|l_i^\times|/|k_i^\times|$ contains no non-trivial $p$-torsion. The condition on $|l_i^\times|/|k_i^\times|$ can be omitted when $p_i=p$.
\end{lem}
\begin{proof}
If $\{l_\alp/k\}_\alp$ is the family of finite subextensions of $l/k$ then $l_i=\cup_\alp(l_\alp)_i$ and $\till_i=\cup_\alp\wt{(l_\alp)}_i$. Hence it suffices to prove the lemma for the extensions $l_\alp/k$, and we can assume in the sequel that $l/k$ is finite.

The opposite implication in (i) is precisely \cite[Proposition~2.2.2]{temst}. Conversely, assume that $l/k$ is unramified, and so $\lcirc/\kcirc$ is essentially \'etale by Lemma~\ref{unramlem}. Since $\lcirc\into l_i^\circ$ and $\kcirc\into k_i^\circ$ are localizations, $l_i^\circ/k_i^\circ$ is essentially \'etale, and hence $l_i/k_i$ is unramified. Let $n=\Ker(\kcirc\to(\tilk_i)^\circ)$ and $m=\Ker(\lcirc\to(\till_i)^\circ)$. Since $\Spec(\lcirc)\to\Spec(\kcirc)$ is a bijection, $m$ is the radical of $n\lcirc$. Since $\lcirc/n\lcirc$ is essentially \'etale over $(\tilk_i)^\circ$, it is reduced and we obtain that $m=n\lcirc$. Thus, $(\till_i)^\circ=\lcirc/m$ is essentially \'etale over $(\tilk_i)^\circ$, and hence $\till_i/\tilk_i$ is unramified.

Let us prove (ii). Since $[l:k]<\infty$, the condition on $|l_i^\times|/|k_i^\times|$ means that $(p,e_{l_i/k_i})=1$. Note that this is automatically so when $l_i/k_i$ is tame and $p_i=p$. It remains to prove the ``if and only if" claim. Our first goal is to replace $l$ and $k$ by $L=l^u$ and $K=k^u$. Note that $l/k$ is tame if and only if $L/K$ is tame. In addition, the extensions $K_i/k_i$, $L_i/l_i$, $\tilK_i/\tilk_i$ and $\tilL_i/\till_i$ are unramified by (i), hence $L_i/K_i$ is tame if and only if $l_i/k_i$ is tame, and $\tilL_i/\tilK_i$ is tame if and only if $\till_i/\tilk_i$ is tame. Thus, it suffices to prove that $L/K$ is tame if and only if both $L_i/K_i$ and $\tilL_i/\tilK_i$ are tame and $(p,e_{l_i/k_i})=1$. Note also that $(\tilK_i)^u=\tilK_i$. Indeed, if there exists a non-trivial unramified extension of $\tilK_i$ then we can lift it to a valued extension $F/K$, and it then follows from (i) that $F/K$ is unramified, which contradicts that $K=K^u$.

Since $K=K^u$, any tame extension is obtained by adjoining roots $a^{1/n}$ with $(p,n)=1$. Hence if $L/K$ is tame then both $L_i/K_i$ and $\tilL_i/\tilK_i$ are tame and $(p,e_{L_i/K_i})=1$.

Conversely, assume that $L/K$ is not tame, in particular, $[L:K]$ is divisible by $p$ and $p>1$. If $L_i/K_i$ is not tame we are done, otherwise $[L:K]=[L_i:K_i]=e_{L_i/K_i}f_{L_i/K_i}$ is divisible by $p$. If $p|e_{L_i/K_i}$ we are done, otherwise $p|f_{L_i/K_i}$. In the latter case, $p=\cha(\tilK_i)$ divides $[\tilL_i:\tilK_i]$. Since $\tilK_i=(\tilK_i)^u$, as was observed above, this implies that $\tilL_i/\tilK_i$ is not tame.
\end{proof}

\begin{cor}\label{tamecritlem}
Let $k$ be a valued field of height $d<\infty$ and let $0\le i\le d$. Set $p=\expcha(\tilk)$ and $p_i=\expcha(\tilk_i)$. Then $k$ is tame if and only if $k_i$ and $\tilk_i$ are tame and $|k^\times_i|$ is $p$-divisible. The condition on $p$-divisibility can be omitted when $p_i=p$.
\end{cor}
\begin{proof}
Note that $(k^a)_i=(k_i)^a$ and its residue field is $(\tilk_i)^a$. So, we can denote these fields $k_i^a$ and $\tilk_i^a$ without any ambiguity. By Lemma~\ref{unramcompos}(ii), $k^a/k$ is tame if and only if both $k_i^a/k_i$ and $\tilk_i^a/\tilk_i$ are tame and $|k_i^\times|$ is $p$-divisible. In addition, if $k_i$ is tame then $|k_i^\times|$ is $p_i$-divisible, so the divisibility condition can be omitted when $p=p_i$. The corollary follows.
\end{proof}

\subsubsection{The tameness theorem}
Now, we can study $p$-closed valued fields of an arbitrary height.

\begin{theor}\label{tameth}
Let $k$ be a valued field of residual characteristic exponent $p$. If $k$ is $p$-closed then it is tame.
\end{theor}
\begin{proof}
Assume first that $k$ is of height $d<\infty$. Since $k$ is $p$-closed, it is perfect. In particular, if $d=0$ then $k$ is tame. The case $d=1$ is covered by Corollary~\ref{tamecor}(ii), so assume that $d>1$ and the theorem is proved for any height smaller than $d$. Choose a natural $i$ with $0<i<d$. Then $k_i$ is obviously $p$-closed and $\tilk_i$ is $p$-closed as otherwise we could lift a $p$-extension of $\tilk_i$ to a $p$-extension of $k$. Thus, $k_i$ and $\tilk_i$ are tame by the induction assumption. In addition, $|k_i^\times|$ is $p$-divisible as otherwise $k$ would have a non-trivial extension $k(a^{1/p})$ of degree $p$. Therefore $k$ is tame by Corollary~\ref{tamecritlem}.

Now, let us prove the general case. Let $l/k$ be a finite extension. Find a subfield $k'\subseteq k$ such that $k'$ is of finite transcendence degree over the prime subfield of $k$ and there exists an extension $l'/k'$ such that $l=l'\otimes_{k'}k$. The first condition implies that $k'$ is of finite height. Replacing $k'$ by its algebraic closure in $k$ we can also achieve that $k'$ is $p$-closed. Then $k'$ is tame by the case we proved first, and hence $l'/k'$ is tame. Therefore $l/k$ is tame and we are done.
\end{proof}

\section{Tame distillation theorems}\label{tamedistsec}

\subsection{Riemann-Zariski spaces}\label{rzsec}
We will use the Riemann-Zariski spaces of pointed schemes introduced in \cite[Section~3.2]{temst}, so let us briefly recall definitions and basic properties. All constructions and results of \cite[Section~3.2]{temst} deal with dominant points $\Spec(K)\to X$ but apply to non-dominant ones as well, so we will work in this greater generality.

\subsubsection{Limits}
In what follows, a {\em pointed scheme} is a morphism $\veps\:\Spec(K)\to X$ with $K$ a field and $X$ a quasi-compact and quasi-separated scheme. Often we will simply write $X$, skipping the point in the notation. A morphism of pointed schemes is a compatible pair of morphisms $j\:\Spec(K')\to\Spec(K)$ and $f\:X'\to X$. A morphism is called a {\em modification} if $j$ is an isomorphism and $f$ is proper (but not necessarily surjective). The {\em Riemann-Zariski space} of a pointed scheme is defined as $\RZ_K(X)=\lim_i X_i$ where the limit is taken in the category of locally ringed spaces over the family of all modifications $X_i\to X$. It is easy to see that the latter family is filtered. In addition, dominantly pointed schemes are cofinal among all modifications, hence our definition is equivalent to the definition at \cite[page 621]{temst}.

Note that $|\RZ_K(X)|=\lim_i|X_i|$ as topological spaces and $\colim_i\pi_i^{-1}\calO_{X_i}=\calO_\gtX$, where $\pi_i\:\gtX=\RZ_K(X)\to X_i$ are the projections. In particular, $\calO_{\gtX,\gtx}=\colim_i\calO_{X_i,\pi_i(\gtx)}$ for any point $\gtx$. A subset $U\subseteq\gtX$ is called {\em constructible} if it is of the form $\pi_i^{-1}(U_i)$ for a constructible $U_i\subseteq X_i$, and constructible subsets form a base of the {\em constructible topology} on $\gtX$.

\subsubsection{The valuative interpretation}
Let $\Val_K(X)$ denote the set of isomorphism classes of morphisms $\gtx:\Spec(\calO)\to X$, where $\calO$ is a valuation ring of $K$ and $\gtx$ extends the point $\Spec(K)\to X$. This is the obvious extension of the definition at \cite[page 622]{temst} to arbitrary pointed schemes.

\begin{lem}\label{rzlem}
Let $\Spec(K)\to X$ be a pointed scheme and $\gtX=\RZ_K(X)$, then

(i) $\gtX$ is compact in the constructible topology and quasi-compact in the usual (Zariski) topology.

(ii) Each stalk $\calO_{\gtX,\gtx}$ is a valuation ring and the correspondence $\gtx\mapsto\calO_{\gtX,\gtx}$ induces a bijection $\RZ_K(X)\toisom\Val_K(X)$.
\end{lem}
\begin{proof}
Each scheme $X_i$ is compact with respect to the constructible topology, hence their limit, which is $\gtX$ with the constructible topology, is compact too. The Zariski topology of $\gtX$ is weaker, hence we obtain (i).

For a detailed proof of (ii) we refer to \cite[Corollaries~3.2.4 and 3.2.5]{temst}. The main idea is that for any $f,g\in\calO_{\gtX,\gtx}$ one of them divides another one, as can be seen by blowing up the ideal $(f,g)$ in some $\calO_{X_i,\pi_i(\gtx)}$. Hence $\calO_{\gtX,\gtx}$ is a valuation ring and the bijectivity of the induced map $\RZ_K(X)\toisom\Val_K(X)$ follows from the valuative criterion of properness.
\end{proof}

We use the bijection $\RZ_K(X)\toisom\Val_K(X)$ to provide the target with the Zariski and constructible topologies.

\begin{rem}
For the sake of completeness we note that the topologies of $\gtX=\RZ_K(X)$ and the sheaf $\calO_\gtX$ can be defined on $\Val_K(X)$ in purely valuation-theoretic terms, see e.g. \cite[Section~2.4, p.78]{insepunif}.
\end{rem}

\subsubsection{Functoriality}
The constructions of $\Val_K(X)$ and $\RZ_K(X)$ are functorial. Indeed, assume that $j\:\eta_L\to\eta_X$, $f\:Y\to X$ is a morphism of pointed schemes $\eta_L\to Y$ and $\eta_K\to X$. It induces a map $\Val_L(Y)\to\Val_K(X)$ by restriction of valuations from $L$ to $K$. Namely, for an element $T=\Spec(\calO)\to Y$ of $\Val_L(Y)$ consider the composition $T\to Y\to X$. Since the generic point of $T$ factors through $\eta_K$, the morphism $T\to X$ factors through the spectrum $T'$ of the valuation ring $\calO\cap K$. This gives an element $T'\to X$ of $\Val_K(X)$.

Let us say that a modification $X_i\to X$ of $\eta_K\to X$ is tight if $\eta_K\to X_i$ is schematically dominant. Tight modifications are cofinal; for example $X_i$ is dominated by the schematic image of $\eta_K$ in $X_i$. In particular, $\RZ_K(X)$ is the limit of all tight modifications of $X$, and we use this in the following paragraph.

For any modification $X_i\to X$ of $\eta_K\to X$ there exists a tight modification $Y_j\to Y$ of $\eta_L\to Y$ such that the composition $Y_j\to Y\to X$ factors through $Y_j\to X_i$. This factorization is unique since $\eta_L\to Y_j$ is schematically dominant and $X_i\to X$ is separated (even proper). Therefore different maps $Y_j\to X_i$ are compatible and we obtain a map of the limits $\RZ_L(Y)\to\RZ_K(X)$. Using the valuative criterion of properness it is easy to see that the two maps we have constructed are compatible with respect to the bijections $\RZ_L(Y)=\Val_L(Y)$ and $\RZ_K(X)=\Val_K(X)$.

\begin{lem}\label{valmaps}
Assume given a morphism of pointed schemes $\eta_L\to Y$ and $\eta_K\to X$ such that the morphism $Y\to X$ is of finite type. Then the induced map $\psi\:\Val_L(Y)\to\Val_K(X)$ is open and continuous both in Zariski and constructible topologies. In particular, it takes constructible sets to constructible sets.
\end{lem}
\begin{proof}
We will use the identifications $\Val_L(Y)=\RZ_L(Y)$ and $\Val_K(X)=\RZ_K(X)$. Let $\{\eta_L\to Y_j\}$ and $\{\eta_K\to X_i\}$ denote the families of tight modifications of $\eta_L\to Y$ and $\eta_K\to X$, respectively. In particular, $X_i$ and $Y_j$ are integral schemes. By $\lam_j\:\RZ_L(Y)\to Y_j$ and $\pi_i\:\RZ_K(X)\to X_i$ we denote the projections.

We can work with bases of topologies, so assume that  $V\subseteq\RZ_L(Y)$ and $U\subseteq\RZ_K(X)$ are open quasi-compact (resp. constructible) sets. Then choosing $i$ and $j$ large enough we can assume that $V$ and $U$ are the preimages of open (resp. constructible) sets $V_j\subseteq Y_j$ and $U_i\subseteq X_i$. Enlarging $j$ we can also achieve that the composed morphism $\eta_L\to\eta_K\to X_i$ factors as $\eta_L\to Y_j\to X_i$ for some morphism $f\:Y_j\to X_i$. This gives a commutative square
$$\xymatrix{
\RZ_L(Y)\ar[r]^(.6){\lam_j}\ar[d]^\psi& Y_j\ar[d]^f\\
\RZ_K(X)\ar[r]^(.6){\pi_i}& X_i
}$$

Since $f^{-1}(U_i)$ is open (resp. constructible), the set $\psi^{-1}(U)=\lam_j^{-1}(f^{-1}(U_i))$ is open (resp. constructible) too. To prove that $\psi$ is open we will first reduce to the case when $K=k(X_i)$ and $f$ is flat.

Let $\{K_l\}$ be the family of finitely generated subextensions of $K/k(X_i)$. For any $i'\ge i$ the morphism $\eta_K\to X_{i'}$ factors through $\Spec(K_l)$ with a large enough $l$. Hence $\RZ_K(X)=\lim_l\RZ_{K_l}(X)$, and it suffices to prove that the maps $\RZ_L(Y)\to\RZ_{K_l}(X)$ are open. Thus, replacing $K$ by $K_l$ we can assume that $K/k(X_i)$ is finitely generated. Then it is easy to see that $\eta_K\to X_i$ possesses a modification $\eta_K\to X_{i'}$ such that $k(X_{i'})=K$. So, enlarging $i$ and $j$ we can assume that $k(X_i)=K$.

Furthermore, by the flattening theorem of Raynaud-Gruson, there exists a blow up $X'\to X_i$ such that the strict transform $Y'$ of $Y_j$ is flat and of finite presentation over $X'$. (One cannot use \cite[Th\'eor\`eme~5.7.9]{RG} directly since $f$ does not have to be of finite presentation, but \cite[tag~081R]{stacks} extends the claim to the case we need. Note that we use here that $f$ is flat and of finite presentation over a non-empty open subscheme of $X_i$ by \cite[tag~052A]{stacks}.) Replacing $X_i$ and $Y_j$ by $X'$ and $Y'$, we can assume that $f$ is flat and of finite presentation. Then $f$ is open in Zariski and constructible topologies, and hence $f(V_j)$ is open (resp. constructible).

The flatness of $f$ implies that for any valuation ring $\calO\in\RZ_K$ centered on a point $x\in X_i$ and a preimage $y\in f^{-1}(x)$, there exists a valuation ring $\calO'$ of $L$ extending $\calO$ and centered on $y$. Indeed, $\calO\otimes_{\calO_{X_i,x}}\calO_{Y_j,y}$ is a domain with field of fractions $k(Y_j)$, hence we can take any valuation ring $\calO'$ centered on a preimage of $y$ in $\Spec(\calO)\times_{X_i}\Spec(\calO_{Y_j,y})$. Thus, the map $\RZ_L(Y)\to\RZ_K(X)\times_{X_i}Y_j$ is surjective, and this implies that $\psi(V)=\pi_i^{-1}(f(V_j))$. Hence $\psi(V)$ is open (resp. constructible).
\end{proof}

\subsubsection{Absolute Riemann-Zariski spaces}\label{absrzsec}
Note that $\Val_K=\Val_K(\Spec(\bfZ))$ is the set of all valuation rings of $K$ and for any pointed scheme $\eta_K\to X$ we have a natural projection $p:\Val_K(X)\to\Val_K$ which is injective if and only if $X$ is separated. We denote the image of $p$ by $\Val_K(X)'$; it is the set of all valuations of $K$ with a center on $X$. Being the image of a compact space, $\Val_K(X)'$ is compact in the constructible topology, and hence quasi-compact in the Zariski topology.

For any extension of fields $L/K$ we denote by $\psi_{L/K}\:\Val_L\to\Val_K$ the induced map. By Lemma~\ref{valmaps}, it is continuous and open.

\subsection{Tame distillation for fields}\label{tamefldsec}

\subsubsection{Tame locus}\label{tamelocussec}
If $L/K$ is an algebraic extension of fields it makes sense to study its tameness with respect to various valuations of $L$. So, we introduce the following terminology. The extension $L/K$ is {\em $\gtx$-tame} for $\gtx\in\Val_L$ if $L/K$ is tame with respect to the valuations of $L$ and $K$ corresponding to $\gtx$ and $\psi_{L/K}(\gtx)$, respectively. The set $\Loc^t_{L/K}$ of all points $\gtx\in\Val_L$ such that $L/K$ is $\gtx$-tame will be called the {\em tame locus} of $L/K$. The complement $\Loc^w_{L/K}=\Val_L\setminus \Loc^t_{L/K}$ will be called the {\em wild locus} of $L/K$.

\subsubsection{Absolute tameness}
If $\Loc^t_{L/K}=\Val_L$ then we say that $L/K$ is {\em absolutely tame}. A field $K$ is {\em absolutely tame } if so is the extension $K^a/K$. For example, it follows from Theorem~\ref{tameth} that any $p$-closed field $K$ of characteristic $p$ is absolutely tame.

\subsubsection{$S$-tameness}\label{Stame}
Let $S\subseteq\Val_L$ be a subset. We say that $L/K$ is {\em $S$-tame} if $S\subseteq \Loc^t_{L/K}$. The main case when we will use this is when $S=\Val_L(X)'$ for a pointed scheme $\Spec(L)\to X$. In this case, we will simply say that $L/K$ is {\em $X$-tame}. Note that if $(\eta_L\to X)\to(\eta_L\to Y)$ is a modification of pointed schemes then $X$-tameness of $L/K$ is equivalent to $Y$-tameness.

\subsubsection{$\calP$-tameness}\label{Ptame}
For any subset $S\subseteq\Val_L$ we define its {\em characteristic} $\cha(S)$ to be the set of non-zero residue characteristics of the elements of $S$. Given a set of primes $\calP$, we say that an algebraic extension $L/K$ is {\em $\calP$-tame} if $\cha(\Loc^w_{L/K})\cap\calP=\emptyset$. In fact, $\calP$-tameness is equivalent to $\Spec(\bfZ_{\calP'})$-tameness, where $\bfZ_{\calP'}$ is obtained from $\bfZ$ by inverting the primes of $\calP'$. In the same way, we say that $K$ is {\em $\calP$-tame} if the extension $K^a/K$ is so. Again, Theorem~\ref{tameth} immediately implies the following result that generalizes our earlier observation on $p$-closed fields:

\begin{lem}\label{Ptamelem}
Any $\calP$-closed field $K$ is $\calP$-tame.
\end{lem}

\begin{rem}
If $K$ is $\calP$-closed it may freely happen that an algebraic (and even finite) extension $L$ of $K$ is not $\calP$-closed. Nevertheless, $L$ is $\calP$-tame because $K$ is $\calP$-tame and subextensions of tame extensions are tame.
\end{rem}

\subsubsection{Unramified locus}
Replacing ``tame" with ``unramified" in Sections \ref{tamelocussec}--\ref{Stame} one obtains definitions of {\em $S$-unramified} extensions and {\em unramified loci}. The latter will be denoted $\Loc^u_{L/K}$.

\begin{lem}\label{tamelocus}
(i) If $L/l/k$ is a tower of algebraic extensions then $$\Loc^t_{L/k}=\Loc^t_{L/l}\cap\psi_{L/l}^{-1}(\Loc^t_{l/k})\ \ \ {\rm and}\ \ \ \Loc^u_{L/k}=\Loc^u_{L/l}\cap\psi_{L/l}^{-1}(\Loc^u_{l/k}).$$

(ii) If $L/k$ is an algebraic extension with subextensions $l/k$ and $K/k$ such that $L=Kl$ then
$$\psi_{L/l}^{-1}(\Loc^t_{l/k})\subseteq \Loc^t_{L/K}\ \ \ {\rm and}\ \ \ \psi_{L/l}^{-1}(\Loc^u_{l/k})\subseteq \Loc^u_{L/K}.$$
\end{lem}
\begin{proof}
The first claim follows from the fact that an extension of valued fields $L/k$ is tame (resp. unramified) if and only if both $L/l$ and $l/k$ are tame (resp. unramified). The second claim follows from the fact that if $l/k$ is tame (resp. unramified) then $lK/K$ is tame (resp. unramified).
\end{proof}

\subsubsection{Openness of the loci}
We have the following openness result for finite extensions:

\begin{lem}\label{openlocus}
For any finite extension $L/K$ both the tame locus and the unramified locus are open subsets of $\Val_L$.
\end{lem}
\begin{proof}
We consider the unramified locus $\Loc^u_{L/K}\subseteq\Val_L$ first. Choose a point $\gty\in\Loc^u_{L/K}$, set $\gtx=\psi_{L/K}(\gty)$, and consider the corresponding valuation rings $\calO_\gty$ and $\calO_\gtx$ of $L$ and $K$, respectively. By our assumption, $\calO_\gty/\calO_\gtx$ is essentially \'etale, hence by Chevalley's theorem \cite[$\rm IV_4$, Theorem~18.4.6(ii)]{ega}, $\calO_\gty$ is a localization of a ring $\calO_\gtx[t]/(f(t))$ such that $f(t)$ is monic and if $a\in\calO_\gty$ denotes the image of $t$ then $f'(a)\in\calO_\gty^\times$. Let $V$ be the set of valuation rings $\calO$ of $L$ such that the coefficients of $f$ lie in $\calO\cap K$, $a\in\calO$ and $f'(a)\in\calO^\times$. Clearly, $V$ is an open neighborhood of $\gty$ and we claim that $V\subseteq \Loc^u_{L/K}$. In other words, if $\gty'\in V$ and $\gtx'=\psi_{L/K}(\gty')$ then $\calO_{\gty'}/\calO_{\gtx'}$ is essentially \'etale.

To prove the claim note that $f(t)\in\calO_{\gtx'}[t]$ and $A=\calO_{\gtx'}[t,f'(t)^{-1}]/(f(t))$ is essentially \'etale over $\calO_{\gtx'}$ by Chevalley's theorem. We claim that the homomorphism $A\to L$ taking $t$ to $a$ is injective. Indeed, it suffices to prove that $f(t)$ generates the kernel of the homomorphism $\calO_{\gtx'}[t]\to L$ taking $t$ to $a$. The latter follows from the facts that $\calO_{\gtx'}$ is integrally closed and $f(t)$ is the minimal polynomial of $a$ over $K$. Note that $A\into L$ factors through $\calO_{\gty'}$ by the very definition of $V$, and hence $\calO_{\gtx'}\to\calO_{\gty'}$ factors through a local homomorphism $\phi\:A_m\to\calO_{\gty'}$, where $m=A\cap m_{\gty'}$. Being $\calO_{\gtx'}$-\'etale, $A$ is integrally closed, and hence $A_m$ is a valuation ring by Lemma~\ref{extlem}. Since $\calO_{\gty'}$ is a local $A_m$-algebra with the same fraction field, $\phi$ is an isomorphism, and we obtain that $\calO_{\gty'}/\calO_{\gtx'}$ is essentially \'etale.

Assume, now, that $\gty$ lies in the tame locus $\Loc^t_{L/K}$. Provide $L$ and $K$ with the valuations corresponding to $\gty$ and $\gtx=\psi_{L/K}(\gty)$, respectively. Since the extension $L^u/K^u$ is tame, it is generated by elements $x_i^{1/n_i}$ with $n_i\in\tilK^\times$, and hence there exist finite unramified extensions $L'/L$ and $K'/K$ such that $L'=K'(x_1^{1/n_1}\..x_r^{1/n_r})$. Let $\gty'\in\Val_{L'}$ and $\gtx'\in\Val_{K'}$ be the corresponding points. The tame locus $\Loc^t_{L'/K'}$ contains all valuations with residue characteristic prime to $n_1\..n_r$, hence it is a neighborhood of $\gty'$. By the first part of the theorem, $\Loc^u_{K'/K}$ is a neighborhood of $\gtx'$. Since $\Loc^t_{L'/K'}\cap\psi_{L'/K'}^{-1}(\Loc^u_{K'/K})\subseteq \Loc^t_{L'/K}$ by Lemma~\ref{tamelocus}(i), $\Loc^t_{L'/K}$ is a neighborhood of $\gty'$. In addition, $\psi_{L'/L}(\Loc^t_{L'/K})\subseteq \Loc^t_{L/K}$ by Lemma~\ref{tamelocus}(i), and using that $\psi_{L'/L}$ is an open map by Lemma~\ref{valmaps}, we obtain that $\Loc^t_{L/K}$ is a neighborhood of $\gty$.
\end{proof}

\subsubsection{Distillation of finite extensions}
Recall that the wild locus $\Loc^w_{L/K}$ was defined in \S\ref{tamelocussec}. More generally, for any subset $S\subseteq\Val_L$ we will use the notation $S^w_{L/K}=S\cap \Loc^w_{L/K}$, and the set $\cha(S^w_{L/K})$ will be called the {\em wild primes set} of $S$ with respect to $K$. For example, if $\cha(K)=p>0$ then $\cha(\Loc^w_{L/K})\subseteq\{p\}$, and if $L/K$ is finite then $\cha(\Loc^w_{L/K})$ is contained in the set of all primes not exceeding $[L:K]$.

\begin{theor}\label{extdistth}
Assume that $L/K$ is a finite field extension, $S$ is a subset of $\Val_L$, which is compact in the constructible topology, and $\calP=\cha(S^w_{L/K})$ is the wild primes set of $S$. Then there exists a finite extension $L'/L$ and a field tower $L'/K'/K$ such that $K'/K$ is a $\calP$-extension and $L'/K'$ is an $S'$-tame extension, where $S'=\psi_{L'/L}^{-1}(S)$. Moreover, if $\oK/K$ is an arbitrary $\calP$-maximal extension then one can achieve that $K'\subseteq\oK$ and $L'=LK'$.
\end{theor}
\begin{proof}
Fix a maximal $\calP$-extension $\oK/K$, choose a composite extension $\oL=L\oK$, and set $\oS=\psi_{\oL/L}^{-1}(S)$. The map $\psi_{\oL/L}$ is compact in the constructible topology, hence $\oS$ is compact in the constructible topology and quasi-compact in the Zariski topology.

In the sequel, all fields are subfields of $\oL$ and the composite extensions are taken inside of $\oL$. We claim that the extension $\oL/\oK$ is $\oS$-tame. Indeed, $\psi_{\oL/L}^{-1}(\Loc^t_{L/K})\subseteq \Loc^t_{\oL/\oK}$ by Lemma~\ref{tamelocus}(ii), hence $\oS^w_{\oL/\oK}\subseteq\psi^{-1}(S^w_{L/K})$ and $\cha(\oS^w_{\oL/\oK})\subseteq\calP$. However, $\oK$ is $\calP$-closed and hence $\oL/\oK$ is $\calP$-tame by Lemma~\ref{Ptamelem}. This proves that $\cha(\oS^w_{\oL/\oK})\cap\calP=\emptyset$ and we obtain that $\oS^w_{\oL/\oK}=\emptyset$.

Given a point $\gtx\in\oS$ provide $\oL$ and its subfields with the corresponding valuation. Since $\oL=L\oK$ is tame over $\oK$, there exists a finite subextension $K_\gtx/K$ of $\oK/K$ such that $L_\gtx=LK_\gtx$ is tame over $K_\gtx$. The tame locus $\Loc^t_{L_\gtx/K_\gtx}$ is open by Lemma~\ref{openlocus}, hence its preimage in $\Val_\oL$ is an open neighborhood of $\gtx$. By the quasi-compactness of $\oS$ there exist finitely many points $\gtx_1\..\gtx_n$ such that $\oS\subseteq\cup_{i=1}^n\psi^{-1}_{\oL/L_i}(\Loc^t_{L_i/K_i})$, where we set $L_i=L_{\gtx_i}$ and $K_i=K_{\gtx_i}$. Let $K'$ be the composite of $K_1\.. K_n$ and set $L'=LK'$ and $S'=\psi_{L'/L}^{-1}(S)$. By Lemma~\ref{tamelocus}(ii), $\psi_{L'/L_i}^{-1}(\Loc^t_{L_i/K_i})\subseteq \Loc^t_{L'/K'}$ and hence $\oS\subseteq\psi^{-1}_{\oL/L'}(\Loc^t_{L'/K'})$. Therefore $S'=\psi_{\oL/L'}(\oS)$ lies in $\Loc^t_{L'/K'}$, that is, $L'/K'$ is $S'$-tame.
\end{proof}

\subsection{Tame distillation for alterations}

\subsubsection{Galois alterations}
Following de Jong we say that an alteration $f\:Y\to X$ is {\em Galois} if $Y^G\to X$ is generically radicial, where $G=\Aut_X(Y)$. This happens if and only if $k(Y)/k(X)$ factors into a composition of a purely inseparable extension $K/k(X)$ and a Galois extension $k(Y)/K$ such that the inclusion $\Aut_X(Y)\into G_{k(Y)/K}$ is an equality. If, in addition, $f$ is generically \'etale then we call it a {\em separable Galois alteration}.

\subsubsection{Tame actions}
Assume that a finite group $G$ acts on a scheme $X$. For any point $x\in X$ let $G_x$ denote the stabilizer of $x$, i.e. the group of elements $g\in G$ that preserve $x$ and act trivially on $k(x)$. Recall that the action is {\em tame} if for each $x$ the order of $G_x$ is invertible in $k(x)$.


\subsubsection{Tame Galois alterations}\label{tamealtsec}
In the sequel, by a {\em Galois covering} we mean a finite separable Galois alteration. A separable Galois alteration (resp. a Galois covering) $f\:Y\to X$ is called {\em tame} if the action of $\Aut_X(Y)$ on $Y$ is tame. In the following lemma, given an integral scheme $X$ and a finite field extension $K/k(X)$, by $\Nor_K(X)$ we denote the normalization of $X$ in $K$, that is, if $X$ is covered by open subschemes $\Spec(A_i)$ then $\Nor_K(X)$ is glued from the spectra of the integral closures of $A_i$ in $K$.

\begin{lem}\label{tamealtlem}
Assume that $X$ is a noetherian universally Japanese integral scheme and $L/k(X)$ is an $X$-tame Galois field extension. Then there exists a modification $X'\to X$ such that $Y'=\Nor_{L}(X')\to X'$ is a tame Galois covering.
\end{lem}
\begin{proof}
Set $G=G_{L/k(X)}$. Let $\{X_i\}$ be the family of all modifications of $X$ and set $Y_i=\Nor_{L}(X_i)$. Note that each $Y_i\to X_i$ is a $G$-Galois covering and let $T_i$ be the set of points of $Y_i$ at which the action of $G$ is not tame. If $X_j\to X$ factors through $X_i$ then the induced morphism $Y_j\to Y_i$ is $G$-equivariant, and hence $T_j$ is contained in the preimage of $T_i$.

We claim that $T_i$ is closed. Since $Y_i$ can be covered by separated (even affine) $G$-equivariant open subschemes, the preimages of the elements of an open affine covering of $X_i$, it suffices to consider the case when $Y_i$ is separated. For a subgroup $H\subseteq G$ let $Y_i^H$ denote the subscheme of points fixed by $H$. It is the intersection of all graphs of translations $h\:Y_i\to Y_i$ with $h\in H$, hence a closed subscheme of $Y_i$. For a prime $p$ let $Y_i(p)$ be the fiber of $Y_i$ over $(p)\in\Spec(\bfZ)$. Then $T_i$ is the union of closed subschemes $Y_i^H\cap Y_i(p)$, where $H\subseteq G$ is a subgroup and $p$ is a prime dividing $|H|$. There are finitely many such pairs $(H,p)$, hence $T_i$ is closed.

We will show by a contradiction that there exists $Y'=Y_j$ such that $T_j=\emptyset$, as required. So, assume that no such $Y_j$ exists. If $Y=\Nor_L(X)$ then the preimage of each $T_j$ is a non-empty closed subset of $\RZ_{L}(Y)$, and by the quasi-compactness of the latter there exists a point $\gty\in\RZ_{L}(Y)$ such that for each $j$ the center $y_j\in Y_j$ of $\gty$ lies in $T_j$. Since $G$ is finite, the groups $G_j=G_{y_j}$ stabilize and we set $G'=\lim_j G_{y_j}$. By our assumption, $p=\cha(k(y_j))$ divides the order of $G'$.

The group $G'$ acts on $\calO_\gty=\colim_j\calO_{y_j}$. The extension $L/k(X)$ is $X$-tame, in particular, it is $\gty$-tame. Thus, the action of $G'$ on $\Spec(\calO_\gty)$ is tame and hence the order of $G_\gty$ is invertible in the residue field $k(\gty)$ of $\calO_\gty$, where $G_\gty$ denotes the subgroup of $G'$ that acts trivially on $k(\gty)$. But $k(\gty)=\colim_j k(y_j)$, and hence $G_\gty=G'$ has order divisible by $p$. The contradiction shows that already the action on some $Y_j$ is tame.
\end{proof}

\subsubsection{The distillation theorem}
For an alteration $f\:Y\to X$ let $\calP^w_{Y/X}$ denote the wild primes set of $\Val_{k(Y)}(X)'$ with respect to $k(X)$. In particular, $\calP^w_{Y/X}$ lies in $\cha(X)$ and its elements are bounded by the degree of $f$.

\begin{theor}\label{altdistth}
Assume that $X$ is a noetherian universally Japanese integral scheme and $Y\to X$ is an alteration, and set $\calP=\calP^w_{Y/X}$. Then there exists an alteration $Y'\to Y$ such that the composed alteration $Y'\to X$ factors into a composition of a tame Galois covering $Y'\to X'$ and a $\calP$-alteration $X'\to X$.
\end{theor}
\begin{proof}
Set $K=k(X)$ and $L=k(Y)$. It is classical that for a fine enough modification $X_1\to X$ one has that $Y_1=\Nor_L(X_1)$ is a modification of $Y$. For example, by the flattening theorem \cite[Th\'eor\`eme~5.7.9]{RG}, there exists a blow up $X_1\to X$ such that the strict transform $Y^\st\to X_1$ of $Y\to X$ is flat. Thus, $Y^\st$ is a modification of $Y$ which is a flat alteration of $X_1$, hence $Y^\st\to X_1$ is finite and $Y_1:=\Nor_L(X_1)=\Nor(Y^\st)$ is a modification of $Y$. So, replacing $X$ and $Y$ with their modifications $X_1$ and $Y_1$ we can assume that $Y=\Nor_L(X)$.

By Theorem~\ref{extdistth} there exists a finite extension $L'/L$ and a field tower $L'/K'/K$ such that $L'/K'$ is $X$-tame and $K'/K$ is a $\calP$-extension. Let $L''$ be the Galois closure of $L'/K'$. By Lemma~\ref{closurelem} below, $L''/K'$ is $X$-tame hence replacing $L'$  by $L''$ we can also achieve that $L'/K'$ is Galois. Since $L'/K'$ is $\Nor_{K'}(X)$-tame, Lemma~\ref{tamealtlem} implies that there exists a modification $X'\to\Nor_{K'}(X)$ such that $Y'=\Nor_{L'}(X')\to X'$ is a tame Galois covering. It remains to note that $Y'\to X$ factors through $Y=\Nor_L(X)$ and hence $Y'\to X'\to X$ satisfy all assertions of the theorem.
\end{proof}

It remains to pay off our debt.

\begin{lem}\label{closurelem}
Let $L/K$ be a finite separable field extension with Galois closure $L'/K$. Let $T\subseteq\Val_K$ be a subset and set $S=\psi^{-1}_{L/K}(T)$ and $S'=\psi^{-1}_{L'/K}(T)$. Then $L/K$ is $S$-tame if and only if $L'/K$ is $S'$-tame.
\end{lem}
\begin{proof}
The direct implication is obvious, so let us prove the inverse one. Let $L_1\..L_n$ denote the subfields of $L'$ that are $K$-isomorphic to $L$. Provide $L'$ with a valuation of $S'$. The induced valuation on $K$ lies in $T$, hence the induced valuations on $L_i$ lie in $S'$ with respect to any choice of a $K$-isomorphism $L_i=L$. In particular, each extension $L_i/K$ is tame and since $L$ is the composite of $L_1\..L_n$, the extension $L/K$ is tame too.
\end{proof}

\section{$\cha(X)$-altered desingularization}\label{desingsec}

\subsection{Terminology}

\subsubsection{Maximal domination}
Assume that $S$ is a noetherian scheme and let $\eta_1\..\eta_n$ be the generic points of the irreducible components of $S$. We call $\eta_S=\coprod_i\eta_i=\Spec(\prod_i k(\eta_i))$ the {\em scheme of generic points} of $S$. A morphism of noetherian schemes $X\to S$ is called {\em maximally dominating} if it takes $\eta_X$ to $\eta_S$. This happens if and only if the generic fiber $X\times_S\eta_S$ is dense in $X$.

\subsubsection{Alterations}\label{altsec}
A morphism $X\to S$ is called an {\em alteration} if it is proper, surjective, maximally dominating, and generically finite. For example, if $S$ is qe and $X$ is the normalization of the reduction of $S$ then $X\to S$ is an alteration in our sense.

\begin{rem}
We consider alterations of general schemes for the sake of formulating resolution results in maximal generality. However, the core case is when a morphism $X\to S$ we want to resolve is between integral schemes, and the general case reduces to this easily.
\end{rem}

\subsubsection{Universal $\calP$-resolvability}
Let $X$ be a noetherian scheme and let $\calP$ be a set of primes. Assume that for any alteration $Y\to X$ and a nowhere dense closed subset $Z\subset Y$ there exists a $\calP$-alteration $f{\colon}Y'\to Y$ such that $Y'$ is regular and $Z'=f^{-1}(Z)$ is an snc divisor. Then we say that $X$ is {\em universally $\calP$-resolvable}. If, in addition, $f$ can be chosen separable then we say that $X$ is {\em universally separably $\calP$-resolvable}.

\subsection{Resolution of relative curves}
In \cite[3.1.6]{X} a finite type morphism $Y\to X$ was called {\em pseudo-projective} if it is a composition of a local isomorphism $Y\to\oX$ and a projective morphism $\oX\to X$. This is a technical notion needed to deal with non-separated morphisms, while for separated morphisms pseudo-projectivity is equivalent to quasi-projectivity. The reader that is only interested in the separated case can assume in the sequel that $f\:Y\to X$ is separated and replace ``pseudo-projective" with ``quasi-projective".

\begin{theor}\label{relcurveth}
Let $S$ be a noetherian qe scheme with scheme of generic points $\eta$, let $f{\colon}X\to S$ be a maximally dominating morphism of finite type, and let $Z\subset X$ be a nowhere dense closed subset. Assume that $S$ is universally $\calP$-resolvable (resp. universally separably $\calP$-resolvable) for a set of primes $\calP$ such that $\cha(X)\subseteq\calP$, and assume that $X_\eta=X\times_S\eta$ is a smooth curve over $\eta$ and $Z_\eta=Z\times_S\eta$ is \'etale over $\eta$. Then there exist a projective $\calP$-alteration (resp. a separable projective $\calP$-alteration) $a{\colon}S'\to S$, a projective modification $b{\colon}X'\to (X\times_SS')^{\rm pr}$, where $(X\times_SS')^{\rm pr}$ is the proper transform of $X$,
\begin{equation*}
\xymatrix{X'\ar[r]^(.35)b\ar[rrd]_{f'} & (X\times_SS')^{\rm pr}\ar@{^{(}->}[r] & X\times_SS' \ar[d]\ar[r]& X\ar[d]^f\\
& & S' \ar[r]^a & S}
\end{equation*}
and divisors $W'\subset S'$ and $Z'\subset X'$ such that $S'$ and $X'$ are regular, $W'$ and $Z'$ are snc, the morphism $f'{\colon}X'\to S'$ is pseudo-projective, $(X',Z')\to(S',W')$ is log smooth, and $Z'=c^{-1}(Z)\cup f'^{-1}(W')$, where $c$ denotes the alteration $X'\to X$.
\end{theor}
\begin{proof}
The particular case of the theorem when $\calP=\{l\}'$ is \cite[Theorem~3.4]{X}, and our proof is very close to the proof of \cite[Theorem~3.4]{X}. In fact, all arguments apply verbatim to our case once one replaces $l'$-alterations with $\calP$-alterations, with the only exception being in step 4. So we only repeat the main line of the proof in \cite[Theorem~3.4]{X} and present the modified step 4 with all details.

Step 0. As in the beginning of the proof of \cite[Theorem~3.4]{X}, we split $Z$ into the union of the horizontal part $Z_h$, which is the closure of $Z_\eta$, and the remaining part $Z_v$, called vertical. Also, as in that proof one notes that if $S_1\to S$ is a (resp. separable) projective $\calP$-alteration and $b_1{\colon}X_1\to(X\times_SS_1)^{\rm pr}$ is a projective modification, then it suffices to prove the theorem for $f_1{\colon}X_1\to S_1$ and the preimage $Z_1\subset X_1$ of $Z$. Thus, in such situation we can replace $f$ by $f_1$. In particular, we can replace $S$ by the reduction of its normalization. Furthermore, we can work separately over each connected component of $S$. Thus, we can assume that $S$ is integral, similarly to \cite[Theorem~3.4]{X}.

Steps 1--2 of \cite[Theorem~3.4]{X} extend verbatim to our situation, and reduce the general case to the following one: $f$ is flat quasi-projective and $Z_h\to S$ is flat. Then step 3 produces a diagram
\begin{equation*}
\xymatrix{\oX \ar[rd]_\of\ar[r] & X\times_S\oS \ar[d]\ar[r]& X\ar[d]^f\\
& \oS \ar[r]^\oa & S}
\end{equation*}
where $\oS$ is normal, $\oa$ is a separable alteration, $\oX\to X\times_S\oS$ is a projective modification and
$(\oX,\oZ_h)$ is a semistable multipointed $\oS$-curve, where $\oZ_h\subset\oX$ is the horizontal part of the preimage $\oZ$ of $Z$.

Step 4. {\it We can assume that $\oa$ is a tame $G$-Galois covering.} First, we note that semistability is preserved by base changes, hence we can enlarge $\oa$ in the sense that we can replace it by any separable alteration $\oS'\to S$ that factors through $\oS$. By Theorem \ref{altdistth}, enlarging $\oa$ we can achieve that it factors into a composition of a tame Galois covering $\oS\to S'$ and a $\calP$-alteration $a\:S'\to S$. By step 0 we can replace $S$ and $X$ by $S'$ and $X'=X\times_SS'$ and then $\oa$ is as required.

The rest is almost the same as in \cite[Theorem~3.4]{X}. By step 5, the action of $G=\Aut_\oS(S)$ canonically lifts from $X\times_S\oS$ to $\oX$. By step 6, the action on $\oX$ is inertia specializing. The latter very mild condition means that if $y\in\oX$ specializes $x\in\oX$ then $G_x\subseteq G_y$, and it is automatically satisfied whenever $\oX$ is covered by $G$-equivariant separated open subschemes, see \cite[3.1.2--3.1.4]{X}. Step 7 of loc.cit. is redundant as $\oa$ is finite by step 4. As in step 8, we choose a closed subset $W\subsetneq S$ such that $\oa$ is \'etale over $S\setminus W$, $\of(\oZ_v)\subset\oW:=\oa^{-1}(W)$, and $\of\:\oX\to\oS$ is smooth over $\oS\setminus\oW$. As in step 9, we apply our assumption on resolvability of $S$ to reduce to the case when $S$ is regular and $W$ is snc. Then, as in step 10, $(\oS,\oW)\to(S,W)$ is a Kummer \'etale morphism of log regular log schemes, and hence the composition $(\oX,\oT)\to(\oS,\oW)\to(S,W)$ is log regular. The group $G$ acts tamely on $(\oX,\oT)$ and trivially on $(S,W)$, hence Gabber's modification theorem \cite[Theorem~3.1.5]{X} applies and produces a modification $(X',T')$ of $(\oX,\oT)/G$ such that $(X',T')\to(X,S)$ is log smooth. It remains to replace $(X',T')$ by a log scheme obtained by applying the monoidal desingularization functor $\tilcalF^\rmlog$ of \cite[3.4.9]{VIII}.
\end{proof}

\subsection{The main theorem}
Now we are in a position to prove the main desingularization result of the paper.

\begin{theor}\label{mainth}
Let $f{\colon}X\to S$ be a maximally dominating morphism of finite type between noetherian qe schemes, let $Z\subset X$ be a nowhere dense closed subset, and assume that $S$ is universally $\calP$-resolvable for a set of primes $\calP$ such that $\cha(S)\subseteq\calP$. Then,

(i) $X$ is universally $\calP$-resolvable.

(ii) There exist projective $\calP$-alterations $a{\colon}S'\to S$ and $b{\colon}X'\to X$ with regular sources, a pseudo-projective morphism $f'{\colon}X'\to S'$ compatible with $f$
\begin{equation*}
\xymatrix{X' \ar[d]^{f'}\ar[r]^b & X \ar[d]^f\\
S' \ar[r]^a & S&}
\end{equation*}
and snc divisors $W'\subset S'$ and $Z'\subset X'$ such that $Z'=b^{-1}(Z)\cup f'^{-1}(W')$ and the morphism $(X',Z')\to(S',W')$ is log smooth.

(iii) If $S=\Spec(k)$, where $k$ is a perfect field, then one can achieve in addition to (ii) that $a$ is an isomorphism and the alteration $b$ is separable. In particular, $X$ is universally separably $\calP$-resolvable in this case.
\end{theor}
\begin{proof}
The theorem generalizes \cite[Theorem~3.5]{X} from the case when $\calP=\{l\}'$. The proof in our case copies the proof of \cite[Theorem~3.5]{X} with the only difference that $l'$ is replaced by $\calP$ and the reference to \cite[Theorem~3.4]{X} in step 6 of the proof is replaced by the reference to Theorem~\ref{relcurveth}. These modifications are absolutely straightforward, so we illustrate them with a single example: in step 4 one should consider an infinite $\calP$-extension $\ok/k$ of a finite field $k$ instead of an $\l'$-extension. This is always possible since $\calP$ contains $\cha(k)$ and hence is not empty.
\end{proof}

\subsubsection{Absolute $\cha(X)$-altered desingularization}
In order to practically use Theorem \ref{mainth} one should start with a class of resolvable schemes. Currently, the best known result for schemes $S$ with a non-empty $\cha(S)$ is desingularization of qe threefolds due to Cossart and Piltant, see \cite{CP}. So, assume that $X$ admits a morphism of finite type $f\:X\to S$ such that $S$ is qe and of dimension at most 3. We can assume that $f$ is dominant, and then the set $\cha(S)\setminus\cha(X)$ is finite. So, replacing $S$ by its open subscheme containing the image of $X$ we can assume that $\cha(X)=\cha(S)$. Since, $S$ is universally resolvable by \cite{CP}, Theorem~\ref{mainth}(i) implies that $X$ is universally $\cha(X)$-resolvable. This proves the first claim of Theorem~\ref{absth}, while the second claim follows from Theorem~\ref{mainth}(iii).

\subsubsection{Desingularization of morphisms}
Once we know that any scheme $S$ of finite type over a qe threefold is universally $\cha(S)$-resolvable, Theorem~\ref{mainth}(ii) implies Theorem~\ref{relth} from the Introduction.

\begin{rem}
One can slightly strengthen the assertion of Theorem~\ref{relth} by requiring that $X'$ is a $\cha(X)$-alteration of a subscheme of $X\times_SS'$ (rather than a $\cha(S)$-alteration). We only outline the main idea, ignoring divisors for simplicity. Factor $X$ through an open subscheme $S_0\subseteq S$ such that $\cha(X)=\cha(S_0)$ and apply Theorem~\ref{relth} to resolve $X\to S_0$ by $X''\to S''_0$. This only involves $\cha(X)$-alterations. Then use the assumption on $S$ to find a $\cha(S)$-alteration $S'_0\to S_0$ that factors through $S''_0$ and embeds into an alteration $S'\to S$ such that $S'$ is regular. Finally, replace $X'$ with $X''\times_{S''_0}S'_0$ and apply the monoidal desingularization functor $\tilcalF^\rmlog$ of \cite[3.4.9]{VIII}.
\end{rem}

\bibliographystyle{amsalpha}
\bibliography{p_alterations}

\end{document}